\documentclass[12pt]{article}
\usepackage{amsmath}
\usepackage{amssymb}
\usepackage{url, enumerate}
\usepackage{graphicx}
\usepackage[T1]{fontenc}

\usepackage{hyperref} 

\newcommand{\Mu}{%
\mathrel{\reflectbox{\rotatebox[origin=c]{180}{$\Omega$}}}}

\newcommand{\eop}{\bigstar}  

\newcommand{\otp}[1]{\hbox{otp($#1$)}}

\newcommand{\cf}{{\rm cf}}

\newenvironment{proof}{\noindent{\bf Proof.}}{\par\bigskip}

\newenvironment{proof-}{\noindent{\bf Proof}}{\par\bigskip}

\newtheorem{THEOREM}{Theorem}[section]

\newtheorem{Conclusion}[THEOREM]{Conclusion}

\newtheorem{Hypothesis}[THEOREM]{Hypothesis}

\newtheorem{LEMMA}[THEOREM]{Lemma}

\newtheorem{Main Theorem}[THEOREM]{Main Theorem}
\newenvironment{main Theorem}{\begin{Main Theorem}} 
{\end{Main Theorem}}

\newtheorem{Theorem}[THEOREM]{Theorem}
\newenvironment{theorem}{\begin{Theorem}}{\end{Theorem}}

\newtheorem{Definition}[THEOREM]{Definition}
\newenvironment{definition}{\begin{Definition}}{\end{Definition}}

\newtheorem{Conventions}[THEOREM]{Conventions}

\newtheorem{Main Definition}[THEOREM]{Main Definition}
\newenvironment{main definition}{\begin{Main Definition}}
{\end{Main Definition}}

\newtheorem{Lemma}[THEOREM]{Lemma}

\newtheorem{Notation}[THEOREM]{Notation}

\newtheorem{Convention}[THEOREM]{Convention}

\newtheorem{Note}[THEOREM]{Note}

\newtheorem{Observation}[THEOREM]{Observation}

\newtheorem{Remark}[THEOREM]{Remark}

\newtheorem{Question}[THEOREM]{Question}

\newtheorem{Main Fact}[THEOREM]{Main Fact}
\newenvironment{main Fact}{\begin{Main Fact}}{\end{Main Fact}}

\newtheorem{Fact}[THEOREM]{Fact}

\newtheorem{Subfact}[THEOREM]{Subfact}

\newtheorem{Claim}[THEOREM]{Claim}

\newtheorem{Main Claim}[THEOREM]{Main Claim}
\newenvironment{main claim}{\begin{Main Claim}}{\end{Main Claim}}

\newtheorem{Crucial Claim}[THEOREM]{Crucial Claim}
\newenvironment{crucial claim}{\begin{Crucial Claim}}{\end{Crucial Claim}}

\newtheorem{Subclaim}[THEOREM]{Subclaim}

\newtheorem{Sublemma}[THEOREM]{Sublemma}

\newtheorem{Corollary}[THEOREM]{Corollary}

\newtheorem{Example}[THEOREM]{Example}

\newtheorem{Problem}[THEOREM]{Problem}

\newtheorem{Proposition}[THEOREM]{Proposition}

\newtheorem{Conjecture}[THEOREM]{Conjecture}

\newtheorem{Discussion}[THEOREM]{Discussion}

\newenvironment{Proof of the Subfact}
{\noindent{\bf Proof of the Subfact.}}{\par\bigskip}

\newenvironment{Proof of the Theorem}
{\noindent{\bf Proof of the Theorem.}}{\par\bigskip}

\newenvironment{Proof of the Proposition}
{\noindent{\bf Proof of the Proposition.}}{\par\bigskip}

\newenvironment{Proof of the Conclusion}
{\noindent{\bf Proof of the Conclusion.}}{\par\bigskip}

\newenvironment{Proof of the Observation}
{\noindent{\bf Proof of the Observation.}}{\par\bigskip}

\newenvironment{Proof of the Fact}
{\noindent{\bf Proof of the Fact.}}{\par\bigskip}

\newenvironment{Proof of the Lemma}
{\noindent{\bf Proof of the Lemma.}}{\par\bigskip}

\newenvironment{Proof of the Claim}
{\noindent{\bf Proof of the Claim.}}{\par\bigskip}

\newenvironment{Proof of the Corollary}
{\noindent{\bf Proof of the Corollary.}}{\par\bigskip}

\newenvironment{Proof of the Subclaim}
{\noindent{\bf Proof of the Subclaim.}}{\par\medskip}

\newenvironment{Proof of the Main Claim}
{\noindent{\bf Proof of the Main Claim.}}{\par\bigskip}

\newenvironment{Proof of the Crucial Claim}
{\noindent{\bf Proof of the Crucial Claim.}}{\par\bigskip}

\newcommand{\hgt}{{\rm ht}}

\newcommand{\rest}{\upharpoonright}  

\newcommand{\deq}{\buildrel{\rm def}\over =}

\newcommand{\FF}{{\mathcal F}}

\newcount\skewfactor
\def\mathunderaccent#1#2 {\let\theaccent#1\skewfactor#2
\mathpalette\putaccentunder}
\def\putaccentunder#1#2{\oalign{$#1#2$\crcr\hidewidth
\vbox to.2ex{\hbox{$#1\skew\skewfactor\theaccent{}$}\vss}\hidewidth}}

\begin{document}

\title{Are all natural numbers the same?}

\author{Mirna D\v zamonja\\
\\
Institut d'Histoire et de Philosophie des Sciences et des Techniques\\ 
CNRS-Universit\' e Paris 1 Panth{\'e}on-Sorbonne\\
13 Rue de Four, 75006 Paris, France\\
and\\
Institute of Mathematics\\
Czech Academy of Sciences\\
{\v Z}itn{\'a} 25115 67 Prague, Czech Republic\\
{\scriptsize {mirna.dzamonja@univ-paris1.fr}, {https://www.logiqueconsult.eu}}}

\maketitle

\begin{abstract} Recently there have been several attempts to develop forcing axioms analogous to the proper forcing axiom (PFA) for cardinals of the form
$\aleph_n$ where $n>1$. We investigate the difficulties of doing this and survey some of the successes.
\footnote{This research was supported by the GA{\v C}R project EXPRO 20-31529X and RVO: 67985840. The author thanks the referee for careful reading of the manuscript and thoughtful remarks.}
\end{abstract}


\section{Introduction} Forcing axioms are statements shown consistent with ZFC using ZFC or ZFC plus some large cardinals
and which imply that the universe of set theory is saturated for a certain kind of generic extension. The prototypical example is Martin Axiom MA$(\kappa)$, which states that for every ccc forcing $\mathbb P$  and every family $\mathcal F$ of $\kappa$ many dense sets in $\mathbb P$, there is a filter of $\mathbb P$ that intersects every set in $\mathcal F$. The consistency of this axiom follows from the consistency of ZFC; the axiom implies that $2^{\aleph_0}>\kappa$ (see \cite{fasttrack} or 
\cite{kunen} for expositions). The advantage of having such an axiom is that rather than doing an argument using iterated forcing every time that one wants to show a consistency result, one can simply assume as a black box that one has the given axiom and proceed to obtain consistency results by purely combinatorial means. Such is for example the approach of the reference book \cite{Fremlin}.

However, many combinatorial statements require more than a simple ccc forcing to be shown consistent. A typical example is the statement that every two $\aleph_1$-dense subsets of the reals are isomorphic, shown consistent by James E. Baumgartner in 
\cite{Baum-alep1d}, but whose consistency does not follow from Martin Axiom, as shown by Uri Abraham and Saharon Shelah in 
\cite{AvrShe-notMA}. Baumgartner's argument used a particular instance of an axiom stronger than MA, namely the proper forcing axiom PFA\footnote{The original paper does not specify an axiom but spells out the specific forcing to use. PFA was discovered afterwards.}. This axiom was proved consistent modulo the existence of a supercompact cardinal, see \cite{fasttrack} for details and the history of this problem. 

If one wishes to obtain the consistency of combinatorial statements about $\aleph_n$ for $n>1$, things become much more difficult. On the level of MA, one can get by using arguments rather analogous to the one used to show the consistency of MA. In this way one proves various  versions of MA generalised to $\aleph_n$-cc forcing which is $(<\aleph_n)$-closed and satisfies some additional properties, such as the ones considered in \cite{Sh80} or \cite{5authorsforcing}. This is very restrictive and certainly does not give any hope of recovering the wealth of the consequences of PFA that are known to hold on $\aleph_1$, for example all the results in the first nine chapters of \cite{Sh_P}. As for generalising PFA itself directly, this cannot be done in the same way that generalisations of MA have been obtained, since PFA decides the value of the continuum to be $\aleph_2$ (this was shown by Stevo Todor{\v c}evi\'c and Boban Veli{\v c}kovi\'c, as reported in 
the book \cite{stevoOCA} and in the paper \cite{Vel_FACA}). Hence we cannot have PFA for families of $\aleph_2$ many dense sets.

One of the most striking results in the theory of forcing in recent years was discovered by Itay Neeman in \cite{Neeman}. He found a new way to iterate proper forcing, using two kinds of models as side conditions and finite supports. This gave a new proof of the consistency (still modulo a supercompact cardinal) of PFA. Doing this iteration with finite supports rather than the classical countable supports gives much more control and opens many directions. This development has resuscitated the hope that proper forcing has analogues for $\aleph_2$ and other values of $\aleph_n$ for $n>1$. Indeed, Neeman's techniques are of such importance that they have already led to a number of applications both in set theory and in set-theoretic topology (see Alan Dow's results in \cite{Dowsideconditions}, Neeman's results in \cite{Neemanapplications} and Veli{\v c}kovi{\'c}'s results in \cite{VeVe}, for example), as well as modifications to obtain further results such as  Veli{\v c}kovi{\'c}'s \cite{Ve17}. Neeman is optimistic about the generalisation of Baumgartner's results to 
$\aleph_2$-dense subsets of the reals. The enthusiasm of the set-theoretic community for this method is such that Neeman obtained the 2019 Hausdorff medal for it. 

In addition to PFA, stronger forcing axioms such as SPFA have been given a new proof, by Moti Gitik and Menachem Magidor in \cite{GitikMagidor}. Variations of Neeman's side conditions have been introduced as {\em virtual models}
in \cite{Bobanvirtual} and as {\em Magidor models} by
Rahman Mohammadpour and Veli{\v c}kovi{\'c}
in \cite{BobanRahman}. The latter notably allows two cardinals to be preserved in the iteration by using two supercompact cardinals at the outset. In another direction, there are higher forcing axioms corresponding to a notion between MA and PFA, obtained by David Asper{\'o} and Miguel Angel Mota, such as \cite{AsperoMotaMA} which 
we discuss in \S\ref{AsperoMota}. 

Listing all these references we have not exhausted the whole elan of the problem of generalising forcing axioms to higher cardinals, where furthermore we should at least mention works that have preceded Neeman's, such as
the very important paper by William Mitchell \cite{Mitch}, followed by our work with Gregor K. Dolinar \cite{DoDz} and several works using his own method by John Krueger, including \cite{Kruegersquare} and \cite{KRUEGER20181044}. There is also a history of using side conditions in forcing, starting by Todor{\v c}evi{\'c}'s paper \cite{Todorcevic_square} and papers leading or parallel to 
Mitchell's work such as Abraham and Shelah's \cite{AbrShe83} and Sy Friedman's \cite{Fried}. Much material for a survey article about higher forcing axioms exists and we certainly hope that there will be such survey articles which will help the non-experts enter this fascinating subject.

Our intention, however, is not to write such a survey article but, to, by insisting on one revolutionary moment through the work of Neeman and one staple axiom PFA, ask one fundamental question which we believe the experts should address. The question is to know where the limits are. Can one extend forcing axioms stronger than MA, specifically PFA, so that they hold at any 
$\aleph_n$? This matter is subject to discussion and even to precision, since there are many ways to interpret the question. In this article we point out some difficulties that one would have to overcome if one wants to have analogues of PFA for cardinals of the form $\aleph_n$ for $n>1$. We also review some known positive results which, under various circumstances, do provide forcing axioms for regular cardinals larger than $\aleph_1$. We include some proofs, because on the one hand these are combinatorial proofs that are possible to write in a relatively short space and, on the other hand, we feel that it is good to have them grouped in one place, as sometimes they are not that easy to find in the original references. In addition, it is often the analysis of the structure of the proofs rather than of the bare statements that leads to new developments and conclusions.

The evidence presented here leads us to form our own view on the question of the limits of our techniques. This is that there are very good reasons that forcing axioms such as PFA cannot be extended as far as $\aleph_4$, but on the other hand, different frameworks can be devised to obtain forcing axioms in the presence of further assumptions, such as the principle 
$\diamondsuit$.

\section{Some obstacles, interleaved with successes}\label{obstacles} 
\subsection{A natural version of the higher MM or SPFA fails} One may wonder why we concentrate on extending PFA to 
$\aleph_2$
rather than being more ambitious and extending the stronger axiom such as MM or SPFA.  We recall that Martin's Maximum MM is the forcing axiom which guarantees that for any forcing notion $\Bbb P$ which preserves stationary subsets of
$\omega_1$ and any family $\FF$ of $\aleph_1$ many dense subsets of $\Bbb P$, there is a generic filter 
in $\Bbb P$ intersecting every element of $\FF$. This notion was introduced by Matthew Foreman, Magidor and Shelah in \cite{FMS}. The axiom SPFA is the analogous statement where the property of preserving stationary subsets of
$\omega_1$ is replaced by semi-properness. This is an earlier notion, introduced by Shelah in \cite{Sh:119}. However, in the
consequent paper \cite{Shelah_SPFA}, Shelah proved that the two axioms are equivalent.

David Asper\'o asked if it is consistent to have a forcing axiom which acts on forcing notions that preserve
stationary subsets of $\aleph_1$ and $\aleph_2$, guaranteeing the existence of a filter which intersects any family of 
$\aleph_2$ sets.
Let us use the following notation:

\begin{definition}\label{strongSPFA} Let $\lambda$ be an infinite cardinal. We let FA${}_\lambda$ stand for the axiom which says that for any forcing notion $\Bbb P$ which preserves stationary subsets of all regular cardinals $\mu\le\lambda$ and any
family $\FF$ of $\lambda$-many dense subsets of $\Bbb P$, there is a filter in $\Bbb P$ which intersects every element of
 $\FF$.
\end{definition}

Hence, Asper\'o's question is if FA${}_{\aleph_2}$ is consistent. Shelah \cite{Sh:784} proved the following striking theorem:

\begin{theorem}\label{Shstrikes} If $\lambda>\aleph_1$ is regular, then FA${}_\lambda$ is false.
\end{theorem}

Theorem \ref{Shstrikes} justifies that most of the effort on extending forcing axioms to higher cardinals has been on PFA. Of course, one may imagine some intermediate axioms which would not fall under the spectre of this theorem. For this article, we feel that PFA is already challenging enough that it suffices to concentrate on it.

\subsection{Diamonds}One of the most important combinatorial principles in set theory is Jensen's
diamond principle \cite{Jen}, $\diamondsuit(\lambda)$ which for a regular uncountable cardinal $\lambda$ states:

There is a sequence $\langle A_\alpha:\,\alpha<\lambda\rangle$ of sets such that for every $\alpha$ we have
$A_\alpha\subseteq \alpha$, and which has the guessing power, in the sense that for every unbounded subset $A$ of $\lambda$ there is a stationary set of $\alpha<\lambda$ such that 
\[
A_\alpha=A\cap \alpha.
\]
It is quite easy to see that $\diamondsuit(\lambda^+)$ implies that $2^\lambda=\lambda^+$. Considering the converse of this implication, it is not true that CH implies that $\diamondsuit=\diamondsuit(\omega_1)$ holds. Several different proofs of this by Shelah can be found in his book \cite{Sh_P}, for example in I\S7. In particular, one can force by proper forcing a model in which CH holds but
$\diamondsuit$ fails. However, in \cite{SheDiamonds} Shelah also proved  the following theorem.

\begin{theorem}[Shelah 2010]\label{She10} Suppose that $\lambda\ge\aleph_1$. Then $2^\lambda=\lambda^+$ implies 
$\diamondsuit(\lambda^+)$. 
\end{theorem}

In particular, Shelah's result shows that we cannot hope to have a direct lifting of proper forcing to $\aleph_n$ for $n>1$, not even to $\aleph_2$. In fact, the most relevant instances of Theorem \ref{She10} very known already quite early on, by the work of John Gregory \cite{Gregory} and Shelah in \cite{Sh108} and are discussed in \cite{Sh_P} as obstacles to generalising proper forcing. 

\medskip

\begin{proof}(sketch) 
We shall use the notation 
\[
S^{\lambda^+}_{<\lambda}=\{\delta<\lambda^+:\, \cf(\delta)<\lambda\}.
\]
Well known work of Kenneth Kunen (see \cite{kunen}) gives many equivalences of $\diamondsuit(\lambda^+)$, using functions in place of sets or allowing the guessing sets to have more elements than just one.
In the spirit of this work, Shelah shows in Claim 2.5 (2) of \cite{SheDiamonds}, that under the assumptions and the notation we have adopted, the statement $\diamondsuit(\lambda^+)$ is implied by  the following statement: 

$(\ast)$ there are sequences $\langle h_\beta:\,\beta<\lambda^+\rangle$ and $\langle u_\delta:\,\delta \mbox{ limit }\in 
S^{\lambda^+}_{<\lambda}\rangle$ such that:
\begin{itemize}
\item $\langle h_\beta:\,\beta<\lambda^+\rangle$ is a family of elements of ${}^{<\lambda^+}\lambda^+$,
\item for every $\delta\in S^{\lambda^+}_{<\lambda}$ we have $u_\delta\subseteq\delta$ and 
$\sup\{|u_\delta|^+:\,\delta\in S^{\lambda^+}_{<\lambda}\}<\lambda$,
\item for every $g\in {}^{\lambda^+}\lambda^+$, the set 
\[
\{\delta\in S^{\lambda^+}_{<\lambda}:\,\sup\{\alpha\in u_\delta:\,g\rest\alpha\in\{ h_\beta:\,\beta\in u_\delta\}\}=\delta\}
\]
is stationary.
\end{itemize}

We shall show that the assumptions of the theorem imply the above restatement of $\diamondsuit(\lambda^+)$.
This is the point of the proof where we use the assumption that $\lambda$ is uncountable, since otherwise limit ordinals $\delta<\lambda^+$ with $\cf(\delta)<\lambda$ would not exist and hence the whole of $S^{\lambda^+}_{<\lambda}$ could not form a stationary set. 

Let $\langle f_\beta:\,\beta<\lambda^+\rangle$ be a listing of ${}^{<\lambda^+}\lambda^+$, which exists by the cardinal arithmetic assumptions.
To each 
$g\in {{}^{\lambda^+}}\lambda^+$ we can associate a function $f_g\in {}^{\lambda^+}\lambda^+$ defined by
\[
f_g(\alpha)=\min\{\beta<\lambda^+:\,g\rest\alpha=f_\beta\}.
\]
We shall equally have a coding of sequences in ${}^\lambda\lambda^+$ (a set which again has size $\lambda^+$ by our assumptions) by elements of $\lambda^+$, called $c$, in the sense that:
\begin{itemize}
\item $c:\,{}^\lambda\lambda^+\to \lambda^+$ is a bijection, and
\item viewing ${}^\lambda\lambda^+$ as a set of sequences, for every element $\bar{s}$ of ${}^\lambda\lambda^+$ 
we have that $c(\bar{s})\ge \sup (\bar{s})$.
\end{itemize}
For every $\varepsilon <\lambda$ we define a function $\pi_\varepsilon:\,\lambda^+\to\lambda^+$, which to a given $\alpha\in \lambda^+$, associates the element $\bar{s}(\varepsilon)$, where $\bar{s}$ is the unique sequence such that $c(\bar{s})=\alpha$.
The functions $\pi_\varepsilon$ will now be combined with the functions  $\langle f_\beta:\,\beta<\lambda^+\rangle$, by defining for every relevant $\varepsilon$ and $\beta$ a function $h_{\varepsilon,\beta}$, with the same domain as $f_\beta$ and such that
\[
h_{\varepsilon,\beta}(\alpha)=\pi_\varepsilon(f_\beta(\alpha)).
\]
For each limit ordinal $\delta\in S^{\lambda^+}_{<\lambda}$,
so an ordinal of size $\le\lambda$, we fix a decomposition 
\[
\delta=\bigcup_{\varepsilon<\lambda} u^\delta_\varepsilon,
\]
where $\langle u^\delta_\varepsilon:\,\varepsilon<\lambda\rangle$ is an increasing continuous sequence of sets, each of size $<\lambda$. 

Now we claim that there exists $\varepsilon^\ast$ such that the sequences $\langle h_\beta=h_{\varepsilon^\ast,\beta}:\,
\beta <\lambda^+\rangle$ and $\langle u_\delta=u^\delta_{\varepsilon^\ast}:\,\delta\mbox{ limit }\in 
S^{\lambda^+}_{<\lambda}\rangle$ can be chosen to witness the version of $\diamondsuit(\lambda^+)$ in the sense of 
$(\ast)$. The proof is by contradiction. So, if the claim is not true, for every $\varepsilon<\lambda$ we can find a function
$g_\varepsilon \in {}^{\lambda^+}\lambda^+$ and a club $C_\varepsilon$ of $\lambda^+$, such that for every
$\delta$ limit in $S^{\lambda^+}_{<\lambda}\cap C_\varepsilon$,
\[
\sup\{\alpha\in u^\delta_\varepsilon:\,g_\varepsilon\rest\alpha\in\{ h_{\beta,\varepsilon}:\,\beta\in u^\delta_\varepsilon\}\}<\delta.
\]
For each $\alpha<\lambda^+$, the sequence $\langle g_\varepsilon(\alpha):\,\varepsilon<\lambda\rangle$
is in ${}^\lambda\lambda^+$, and hence we have defined $c(\langle g_\varepsilon(\alpha):\,\varepsilon<\lambda\rangle)$
as some ordinal $<\lambda^+$. We define a function $g:\,\lambda^+\to\lambda^+$ by
\[
g(\alpha)=c(\langle g_\varepsilon(\alpha):\,\varepsilon<\lambda\rangle).
\]
At the beginning of the proof, to this function we have associated the function $f_g:\lambda^+\to\lambda^+$. Clearly,
the set 
\[
C=\{\delta\mbox{ limit }\in S^{\lambda^+}_{<\lambda}:\,(\forall \alpha <\delta) f_g(\alpha)<\delta\}
\]
is a restriction of a club of $\lambda^+$ to $S^{\lambda^+}_{<\lambda}$. Let $E=\bigcap_{\varepsilon<\lambda} C_\varepsilon\cap C$, again a club of $\lambda^+$ restricted to $S^{\lambda^+}_{<\lambda}$.
Let $\delta^\ast\in E$. Since $\delta^\ast\in E$ we have that for every $\alpha<\delta^\ast$ also $f_g(\alpha)<\delta^\ast$.
Checking against the definition of $f_g$, this means that for all $\alpha<\delta^\ast$ there is $\beta<\delta^\ast$ satisfying 
$g\rest\alpha=f_\beta$. This $\beta$, as well as $\alpha$, belong to some $u^{\delta^\ast}_\varepsilon$, since these sets union up to
$\delta^\ast$. For each $\alpha<\delta^\ast$, choose the minimal such $\varepsilon$ and call it $\varepsilon^\ast_\alpha$.
Since $\cf(\delta^\ast)<\lambda$, there is $\varepsilon^\ast$ such that the set $B=\{\alpha<\delta^\ast:\,\varepsilon^\ast_\alpha<\varepsilon^\ast\}$ is unbounded in $\delta^\ast$. Since the sets $u^{\delta^\ast}_\varepsilon$
are increasing with $\varepsilon$, we have that for all $\alpha\in B$, the values $\alpha,f_g(\alpha)$ are in $u^{\delta_\ast}_{\varepsilon^\ast}$.

Then for every $\alpha\in B$ we have by the definition of $f_g$ that $g\rest \alpha=f_{f_g(\alpha)}$, which implies that for all
$\varepsilon <\lambda$ we have $g_\varepsilon\rest\alpha= h_{f_g(\alpha), \varepsilon}$, and this is in particular 
true for $\varepsilon^\ast$. However, we have chosen $\delta^\ast\in C_{\varepsilon^\ast}$ and we have a contradiction.
$\eop_{\ref{She10}}$
\end{proof}

\subsection{Strong Chains} One of the great successes of proper forcing is the study of the structure ${}^\omega\omega$ ordered by eventual domination. Many cardinal invariants of this structure have been studied and various consistent inequalities obtained, typically by making one of the invariant equal to $\aleph_1$ and the other equal to $\aleph_2$. This subject became known as the {\em set theory of the reals}. A detailed presentation of the techniques can be found in the book  
\cite{BartoszynskiJudah} by Tomek Bartoszy\'nski and Haim Judah. Lifting such results one step, such as to the structure
of ${}^{\omega_1}\omega_1$ modulo countable is now somewhat understood, through the work in generalised descriptive set theory, in this instance already started by Alain Mekler and Jouko V\"a\"an\"anen in \cite{MekVa}. A
challenge that one would hope to overcome by any axiom that would claim to be an analogue to PFA to $\aleph_2$ is to be able to study the structure of ${}^{\omega_1}\omega_1$ modulo {\em finite}, hence leaving a one cardinal gap between the size of the negligible sets in the order (finite) of the functions, and the size of the their domain ($\aleph_1$).
 
 A pioneering step in this programme was provided by Piotr Koszmider in \cite{Koszmider}, where he proved the following
 
\begin{theorem}[Koszmider 2000] \label{Koszmider} It is consistent that there is an $\omega_2$-sequence of functions in 
${}^{\omega_1}\omega_1$ which is strictly increasing modulo finite. 
\end{theorem}

Koszmider used forcing with side conditions coming from a morass. His result was reproved by Veli{\v c}kovi{\'c} in \cite{VeVe} as one of the first applications of Neeman's method of forcing with two types of side conditions. However, if one wants to lift this result one cardinal up, so increasing the gap to two cardinals, it is not possible; Shelah \cite{Shelah-flat} proved the following theorem:

\begin{theorem}[Shelah 2010]\label{anotherShe10} There is no sequence of length $\omega_3$ of functions in ${}^{\omega_2}\omega_2$ which is strictly increasing modulo finite.
 \footnote{In unpublished work Tanmay Inamdar showed a similar result regarding the increasing sequence of sets: there
 is no sequence of length $\omega_3$ of sets in $[\omega_2]^{\omega_2}$ which are strictly $\subseteq$-increasing modulo finite. This seems to be a genuinely different statement than that of Theorem \ref{anotherShe10}.}
  \end{theorem}
  
 \begin{proof} Suppose for a contradiction that we have a sequence $\langle f_\alpha:\, \alpha <\omega_3\rangle$ of functions in ${}^{\omega_2}\omega_2$ which is strictly increasing modulo finite.
 
Let us observe that there is a cofinal subset $\mathcal F$ of $[\omega_2]^{\aleph_0}$ of size $\aleph_2$. Namely for every $\mu\in[\omega_1,\omega_2)$, we can fix a bijection $e_\mu:\,\omega_1\to \mu$. Let 
\[
\FF=\{e_\mu\,'' i :\, i<\omega_1, \mu<\omega_2\}.
\]
If $A\in [\omega_2]^{\aleph_0}$, there is $\mu<\omega_2$ such that $A\subseteq \mu$. Moreover, since $A$ is countable,
there is  $i<\omega_1$ such that $A\subseteq e_\mu\,'' i$.
Now for every $F\in\FF$ and $\alpha<\omega_3$ we define $f^F_\alpha:\omega_2\to\omega_2\cup\{\infty\}$ by letting
\[
f^F_\alpha(\mu)=
\begin{cases}
\min(F\setminus f_\alpha(i)) &\mbox{ if defined}\\
\infty&\mbox{ otherwise}.\\
\end{cases}
\]
For every $\nu<\omega_2$ we let $f^{F,\nu}_\alpha=f^F_\alpha\rest [\nu, \nu + \omega)$.

Fixing $F\in \mathcal F$ and $\nu < \omega_2$, we have obtained a sequence 
$\langle f^{F,\nu}_\alpha:\,\alpha<\omega_3\rangle$ of functions from $[\nu, \nu + \omega)$ to $F\cup\{\infty\}$ which is 
increasing modulo finite. It is to be expected that there will be some repetition in this sequence. In anticipation, let us define
for every $F\in\FF$ the set $B_F$ to be the set of all $\nu<\omega_2$ such that $\langle f^{F,\nu}_\alpha:\,\alpha<\omega_3\rangle$ is not eventually constant modulo finite. For every $F\in\FF$ and $\nu<\omega_2$ we can choose a club $C^F_\nu$ of $\omega_3$ such that:
\begin{itemize}
\item if $\nu\in B_F$, then for every $\beta<\gamma$ both in $C^F_\nu$ we have that $f^{F,\nu}_\beta$ is strictly below 
$f^{F,\nu}_\gamma$ modulo finite,
\item if $\nu\notin B_F$, then $\langle f^{F,\nu}_\alpha:\,\min(C^F_\nu)\le\alpha<\omega_3\rangle$ is constant modulo finite.
\end{itemize}
Since $\mathcal F$ is of size $\aleph_2$, the above collection of clubs of $\omega_3$ has size $\aleph_2$ and hence
\[
C\deq \bigcap\{ C^F_\nu:\,F\in \mathcal F, \nu<\omega_2\}
\]
is a club of $\omega_3$. Let $\langle \alpha_i:\;i<\omega_1\rangle$ be any increasing sequence of elements of $C$.
Therefore for any $i<j<\omega_1$, the set
\[
u_{i,j}\deq\{\nu<\omega_2:\,f_{\alpha_i}(\nu)\ge f_{\alpha_j}(\nu)\}
\]
is finite. Therefore $\bigcup_{i,j<\omega_1} u_{i,j}$ is a subset of $\omega_2$ of size $\aleph_1$ and, hence, it is contained in $\nu$ for some $\nu<\omega_2$. It follows that the interval $[\nu, \nu+\omega)$ is disjoint from 
$\bigcup_{i,j<\omega_1} u_{i,j}$. This means that for every $o\in [\nu, \nu+\omega)$, the sequence
\[
\langle f_{\alpha_i}(o):i<\omega_1\rangle
\]
is strictly increasing. By the choice of $\FF$, there is $F\in \FF$ such that 
\[
\{f_{\alpha_0}(o), f_{\alpha_1}(o):\, o\in[\nu, \nu+\omega) \}\subseteq F.
\]
Hence for every $o\in [\nu, \nu+\omega)$ we have
\[
f_{\alpha_0}(o)< f^F_{\alpha_0}(o)=f^{F,\nu}_{\alpha_0}(o)\le f_{\alpha_1}(o)< f^F_{\alpha_1}(o)= f^{F,\nu}_{\alpha_1}(o).
\]
In particular, $f^{F,\nu}_{\alpha_0}(o)< f^{F,\nu}_{\alpha_1}(o)$. This shows that $f^{F,\nu}_{\alpha_0}$ and 
$f^{F,\nu}_{\alpha_1}$ are not
equal modulo finite. By the choice of $C$, it follows that $\nu\in B_F$ and hence the sequence
$\langle f^{F,\nu}_{\alpha_i}:\,i<\omega_1\rangle$ is strictly increasing modulo finite. In particular, for every 
$i<\omega_1$ there is $o_i \in [\nu, \nu+\omega) $ such that $f^{F,\nu}_{\alpha_i}(o_i)< f^{F,\nu}_{\alpha_{i+1}}(o_i)$, which implies that there is $\zeta_i\in F$ such that $f^{F,\nu}_{\alpha_i}(o_i)\le \zeta_i < f^{F,\nu}_{\alpha_{i+1}}(o_i)$.
In particular, 
\[
f_{\alpha_i}(o_i)\le \zeta_i<f_{\alpha_{i+1}}(o_i).
\]
Both the interval $[\nu, \nu+\omega) $ and the set $F$ are countable, so there must be $j<k<\omega_1$ and $o^\ast, \zeta^\ast$ such that
\[
o_j=o_k=o^\ast\mbox{ and } \zeta_j=\zeta_k=\zeta^\ast.
\]
Since $\langle f_{\alpha_i}(o^\ast):\, i<\omega_1\rangle$
is strictly increasing, we have 
\[
f_{\alpha_j}(o^\ast)\le \zeta^\ast< f_{\alpha_{j+1}}(o^\ast)\le f_{\alpha_k}(o^\ast)\le\zeta^\ast,
\]
which is a contradiction.
$\eop_{\ref{anotherShe10}}$
\end{proof}

\subsection{Club guessing}
Another combinatorial principle that distinguishes
$\omega_1$ and $\omega_2$ is the {\em club guessing}, a ZFC principle\footnote{By the expression `ZFC principle' we mean a combinatorial statement true in ZFC.}
that was discovered by Shelah in his development of pcf theory, see \cite{Sh_CA}. The statement of club guessing is
similar to $\diamondsuit$, but only club sets are guessed. Many versions of this principle are to be found in \cite{Sh_CA}, we give the version we find the simplest.  We use the notation $S^2_0$ for $\{\delta< \omega_2: \, \cf(\delta)={\aleph_0}\}$.
The following is an auxiliary definition.

\begin{definition}\label{candidates} A {\em candidate} is a sequence 
\[
\langle C_\delta:\,\delta\in S\rangle
\]
where $S$ is stationary subset of $S^{2}_{0}$ and
each $C_\delta$ is an unbounded subset of the corresponding $\delta$ with $\otp{C_\delta}=\omega$.
\end{definition}

\begin{theorem}[Shelah]\label{clubguessing}  
There is a candidate which {\em guesses clubs of $\omega_2$}, that is such that for every club $C$ of $\omega_2$, satisfies
\[
\{\delta:\, C_\delta\subseteq C\} \mbox{ is stationary}.
\]
\end{theorem}

\begin{proof} We shall perform a construction by induction on $\zeta<\omega_2$.

At the stage 0, let $\langle C_\delta^0:\,\delta\in S^2_0\rangle$ be any candidate. If it guesses clubs, we are
done. Otherwise, there is a club $C^0$ of $\omega_2$ such that $\{\delta:\, C_\delta^0\subseteq C^0\}$ is non-stationary,
and therefore, there is a club $D^0$ of $\omega_2$ such that 
\[
\delta\in D^{0}\cap S^2_0 \implies C^0_\delta\setminus C^0\neq \emptyset.
\]
Let $E_0=\lim(C^0\cap D^0)$, so still a club of $\omega_2$. Let $S_0=\lim(E_0)\cap S^2_0$.

Given a club $E_\zeta$ of $\omega_2$ and a candidate  $\langle C_\delta^\zeta:\,\delta\in S_\zeta\rangle$ which 
does not guess clubs of $\omega_2$, we replace $\langle C_\delta^\zeta:\,\delta\in S_\zeta\rangle$
by another candidate, 
\[
\langle C_\delta^{\zeta+1}:\,\delta\in S_{\zeta+1}\rangle
\]
as follows. We first find a club $C^{\zeta+1}$
of $\omega_2$ which exemplifies that $\langle C_\delta^\zeta:\,\delta\in S_\zeta\rangle$ does not guess clubs, so such that there is a club $D^{{\zeta+1}}$ with 
\[
\delta\in D^{\zeta+1} \cap S_\zeta  \implies C^\zeta_\delta\setminus C^{\zeta+1} \neq \emptyset.
\]
Let $E_{\zeta+1}=\lim(C^{\zeta+1}\cap D^{\zeta+1}\cap E_\zeta)$, so still a club of $\omega_2$. Let 
$S_{\zeta+1}=\lim(E_{\zeta+1})\cap S_\zeta$, so still a stationary subset of $S^2_0$. For $\delta\in  S_{\zeta+1}$, we let
\[
C_\delta^{\zeta+1}=\{ \sup(\alpha\cap E_{\zeta+1}):\,\alpha\in C_\delta^\zeta \setminus \min (E_{\zeta+1})\}.
\]
Notice that for $\delta\in  S_{\zeta+1}$, we have that $\delta\in \lim (E_{\zeta+1})$, so $\sup(C_\delta^{\zeta+1})=\delta$, and also note that $C_\delta^{\zeta+1}\subseteq E_{\zeta+1}$ and $E_{\zeta+1}\subseteq E_{\zeta}$. 

For $\zeta<\omega_2$ limit ordinal let $E_\zeta=\bigcap_{\xi<\zeta} E_\xi$, still a club of
$\omega_2$. Let $S_\zeta=S^2_0\cap \lim(E_\zeta)$. In particular, for every $\xi <\zeta$ we have that $E_\zeta\subseteq E_\xi$ and observe that also $S_\zeta\subseteq S_\xi$. Let $\langle C_\delta^{\zeta}:\,\delta\in S_{\zeta}\rangle$ be any candidate. If it guesses clubs, we are done. Otherwise, we can continue the induction.
 
Assuming the induction has gone to the end, let us take any $\delta^\ast\in S_{\omega_1}$, hence $\delta^\ast\in\lim(E_{\omega_1})\cap S^2_0$
and in particular $\delta^\ast$ is a limit ordinal larger than $\min(E_{\omega_1})$. Notice that for any 
$\alpha>\min(E_{\omega_1})$, we have that $\langle  \sup(\alpha\cap E_{\zeta}):\,\zeta<\omega_1\rangle$ is a non-increasing sequence of ordinals, so it must be
eventually constant. For each such $\alpha$ let $\zeta_\alpha<\omega_1$ be an ordinal after which 
 $\langle  \sup(\alpha\cap E_{\zeta}):\,\zeta<\omega_1\rangle$ is constant. Let 
  \[
 \zeta^\ast=\sup_{\alpha\in {\delta^\ast}
 \setminus \min(E_{\omega_1})}\zeta_\alpha.
 \]
 Then we have that $\zeta^\ast<\omega_2$ and
 $C_{\delta^\ast}^{\zeta^\ast+2}=C_{\delta^\ast}^{\zeta^\ast+1} \subseteq E_{\zeta^\ast+1}$, which is in a contradiction with the choice of 
 $\langle C_\delta^{\zeta^\ast+2}:\,\delta\in S_{\zeta^\ast+2}\rangle$. This contradiction shows that the induction must have stopped 
 at some moment, so we must have found a sequence that guesses clubs.
$\eop_{\ref{clubguessing}}$
\end{proof}

The situation at $\omega_1$ is different. Namely, it is consistent that there is no club guessing between $\omega$ and
$\omega_1$. In particular PFA implies that there is no such guessing. This was proved by Shelah, see [XVII 3.14.A ,
\cite{Sh_P}].

\begin{theorem}[Shelah]\label{PGAdoesnotguess} PFA implies that there is no sequence 
$\langle C_\delta:\,\delta\in S\rangle$ for some $S$ stationary subset of the set of countable limit ordinals such that 
each $C_\delta$ is an unbounded subset of the corresponding $\delta$ with $\otp{C_\delta}=\omega$, and such that
for every club $C$ of $\omega_1$,
\[
\{\delta:\, C_\delta\subseteq C\} \mbox{ is stationary}.
\]
\end{theorem}

Putting the two above theorems together, we see that whatever a generalisation of PFA we may decide to adopt at 
$\omega_2$, it will
not be able to deal with the same kind of forcing as the one we are used to having at $\omega_1$. 

\subsection{Trees, a quick overview}
Further combinatorial differences exist between $\aleph_1$ and $\aleph_2$, perhaps the most well known one is that on 
$\omega_1$ we have special Aronszajn trees, as reported in \cite{Kurepathesis}, while, ever since the celebrated work of Mitchell \cite{Mitchellmodel}, there are consistency results showing that the analogue is not true on
$\omega_2$. Indeed, the whole large chapter of combinatorics of $\omega_1$ is built on fine understanding of Aronszajn and Souslin trees, starting from  {\DJ}uro Kurepa's work in \cite{Kurepathesis}. It would even be preposterous to include concrete references to justify this statement, since there are so many that quoting a few would only have the effect of excluding many others, but let us mention some favourites such as Todor{\v c}evi{\'c}'s \cite{Stevoontrees}
and Justin Moore's \cite{JustinstructureA}.

On $\omega_2$ this is a completely different matter. For example, the celebrated paper of Richard Laver and Shelah \cite{LaverShelah} proves it consistent from a weakly compact cardinal that there are no $\aleph_2$-Souslin trees and CH holds, but to this day the question that they have asked of the consistency of their result with GCH, is still open.
A recent paper by Chris Lambie-Hanson and Assaf Rinot \cite{Lambie-Henson-Rinot} indicates that this might not even be possible, since it proves that if $\lambda^{++}$ is not a Mahlo cardinal in $\mathbf L$, then $2^\lambda=\lambda^+$ implies the existence of a 
$\lambda^{++}$-Souslin tree. In fact, Lambie-Hanson and Rinot develop an interesting forcing axiom on their way to this result. Their forcing axiom gives a framework for constructing objects of size $\lambda^+$ from $\diamondsuit(\lambda)$
and a weakening of $\square_\lambda$.

Tree combinatorics is a very rich area which we obviously cannot cover in detail. A survey of some of the main results in the context of forcing axioms can be found in \cite{fasttrack}. We feel that a fine understanding of combinatorics of trees on $\aleph_2$ 
is now within our reach thanks to the developments both on forcing with two types of side conditions and related methods, and to the axioms such as the one by Lambie-Henson and Rinot. 
To mention one work in progress, let us discuss an interesting concept that makes as much sense at $\aleph_2$ as it does on $\aleph_1$, which is the idea of a wide Aronszajn tree. 

A {\em wide Aronszajn tree} is a tree of height and size $\aleph_1$ but with no uncountable branches. Motivated by their use of such trees in abstract model theory, Mekler and V\"a\"an\"anen asked in their 1994 paper \cite{MekVa} if under 
MA$(\omega_1)$ there is one such tree into which all the others embed under strict order preserving embeddings. This question was resolved in our recent paper with Shelah \cite{wideAronszajn}, where we showed that MA$(\omega_1)$ implies that there is no such tree. The concept of a wide $\aleph_2$-Aronszajn tree is very natural: a tree of height and size 
$\aleph_2$ with no branches of length $\omega_2$, and it is easily seen that such trees exist just in ZFC. To study their universality under strict order preserving embeddings it no longer suffices to use MA, but consistency results using the method of two type of models as side conditions are studied in our work in progress \cite{Mirnaonthetree}.

To wrap this up, let us mention a related concept that saw one of the first applications of Neeman's method, given by Neeman himself in
\cite{Neemanapplications}. It is a somewhat technical result about {\em weakly special $\aleph_2$-Aronszajn trees}. An 
$\aleph_2$-Aronszajn $T$ is weakly special if there is a function $f$ from $T$ into the ordinals which is
\begin{itemize}
\item injective on chains, and satisfies
\item $f(t)<\hgt(t)$ for every $t\in T$.
\end{itemize}
Adding such a weak specialising function has a natural impediment, as explained in Claim 1.2. of \cite{Neemanapplications}.
The main theorem of \S4 of  \cite{Neemanapplications} shows that modulo such an impediment, ({`}extensivelly overlapped nodes'), a weak specialising 
function on a club of levels of an $\aleph_2$-Aronszajn  tree $T$, can be added using forcing with two kinds of models as side conditions.  

\subsection{The size of the continuum}\label{AsperoMota} Having seen some of the combinatorial obstacles to generalisations of PFA, we can now go back to basics: if we have an analogue of PFA for $\aleph_2$, what should the size of the continuum be in the resulting model? This axiom would
for example solve the Baumgartner problem for $\aleph_2$-dense sets of reals, that is prove that it is consistent modulo large cardinals that the continuum is at least $\aleph_3$ and every two $\aleph_2$-dense sets of reals are isomorphic. Neeman has a work in progress where he is believed to have solved this particular problem. He has solved several other specific problems
(see his work with Thomas Gilton in \cite{GiltonNeeman2}) too but an axiom solving these and many other questions for which we do not know if they are consistent with $2^{\aleph_0}=\aleph_3$, is another matter. The point is that, apart from MA, it is very hard to find axioms consistent with `the large continuum', meaning simply $2^{\aleph_0}\ge\aleph_3$. This problem is very old and was specifically stated in Shelah's ``Logical Drreams''\cite{Sh:E23} and re-iterated in \cite{Sh:E73}.

One obvious remark to make is that if indeed this analogue of properness is found for $\aleph_2$, call it $\aleph_2$-properness, then not every proper forcing can be $\aleph_2$-proper, since then the forcing axiom for $\aleph_2$-proper posets would imply PFA, which decides that the value of the continuum is $\aleph_2$. In pre-Neeman generalisations of proper forcing, mostly done by Andrzej Ros{\l}anowski and Shelah, such as in \cite{RoSh655} and others in that series of articles, generalisations of properness were proposed. These works concentrated on forcing which is $(<\lambda)$-closed, and has some analogue of properness at $\lambda$. In this scenario, the $(<\lambda)$-closure is used to preserve cardinals up to $\lambda$, so this does not address the question of the large continuum. Newer method's using side conditions and preserving two cardinals at a time can potentially overcome this problem. More recent work on iterations making the continuum large is by Asper{\'o} and Mota in \cite{AsperoMotaMA}, where they obtain a forcing axiom that they call
MA${}^{1.5}$. This axiom has many pleasing properties in obtaining certain consequences of PFA consistent with the continuum large. Most importantly, as Asper{\'o} and Mota showed in a further paper \cite{AsperoMotaMA2}, MA${}^{1.5}$ negates Moore's principle $\Mu$. This principle states:

$\Mu$: there is a sequence $\langle f_\xi:\,\xi<\omega_1\rangle$ of continuous
functions $f_\xi:\,\xi\to\omega$ such that for every club $C$ of $\omega_1$ there is some $\xi\in C$ satisfying that 
for each $k<\omega$, the set $f^{-1}(k)\cap C\cap \xi$ is unbounded in $C\cap\xi$.

It is still not known if certain weaker variants of this principle can fail in models where the continuum is larger than 
$\aleph_2$.

Another difficult question is to have a forcing axiom consistent with CH. Much on this is said in Shelah's book \cite{Sh_P}, for example in VII \S2, VII \S3 and XVIII \S1, and several different methods are proposed there. None satisfies what Shelah has put as a gold standard of iteration: the matching of the property $\varphi$ being iterated with the way that the iteration is done, so that the iteration preserves 
$\varphi$. Asper{\'o} and Mota studied this question through the method of symmetric models, in a series of papers, still in development.

\section{Conclusion} Clearly, Neeman's method and its relatives, many but not all of which we mentioned here, are 
very promising. They have revolutionised the way that we see iteration of forcing. Some other forcing axioms exist as well and many are in development. However, at this moment it is not clear how far they can go. Various combinatorial results show that $\aleph_1$ is after all, somewhat special. Perhaps it is because it is the successor of an `inaccessible' cardinal, which is not the case for any other $\aleph_n$. For whatever reason, not all natural numbers are the same when it comes to infinite combinatorics, as the numerous `obstacles' above show.

 At least to us, it seems that a serious upper bound to the natural development of forcing axioms presents itself at $\aleph_4$, if not before. By having this opinion we differ from some others in the field; for example Neeman feels that there is no real reason that the method should stop at any natural number (Neeman's Hausdorff Medal winning lecture, Vienna 2019). 
   
 Be it one way or another, it seems clear that the development of generalisations of proper forcing and forcing axioms in general is one of the most interesting questions in set theory at the moment and that it will bring us new understanding of infinite combinatorics. Moreover, it looks like such a development must pass through a fine analysis of the obstacles mentioned in the above.

\bibliographystyle{plain}
\bibliography{../bibliomaster}

\end{document}